\documentclass{amsart}

\usepackage{amsmath,amsfonts,amscd,amsthm}
\usepackage[utf8]{inputenc}
\usepackage{float}
\usepackage{todonotes}
\usepackage{tkz-graph}
\usepackage[pagebackref,bookmarksopen]{hyperref}
\hypersetup{
	colorlinks=true,
	linkcolor=blue,
	citecolor=magenta,
	urlcolor=cyan,
}
\newcommand{\C}{\mathbb{C}} % amsfonts

\newcommand{\matr}[4]{\ensuremath{{\begin{pmatrix}#1&#2\\#3&#4\end{pmatrix}}}}
\newcommand{\smatr}[4]{\ensuremath{\bigl(\begin{smallmatrix}
      #1&#2\\ #3&#4
      \end{smallmatrix}\bigr)}}
\newcommand{\Pj}[1]{\ensuremath{\mathrm{P}_{#1}}}
\newcommand{\mon}[2]{\ensuremath{x^{#1}y^{#2}}}

\DeclareMathOperator{\trace}{Trace}

\newtheorem{theorem}{Theorem}[section]

\newtheorem{corollary}[theorem]{Corollary}
\newtheorem{proposition}[theorem]{Proposition}
\theoremstyle{remark}
\newtheorem{remark}[theorem]{Remark}

\begin{document}

\title[Powers and trace of symmetric powers of $2\times 2$ matrices]{Powers and trace of symmetric powers of $2\times 2$ matrices and combinatorial, Fibonacci and Lucas identities}
\author[J.~L.~Cisneros-Molina]{Jos\'e Luis Cisneros-Molina}
\address{Instituto de Matem\'aticas, Unidad Cuernavaca, Universidad Nacional Aut\'onoma
de M\'exico, Avenida Universidad s/n, Colonia Lomas de Chamilpa, Cuernavaca,
Morelos, Mexico}
\curraddr{}
\email{jlcisneros@im.unam.mx}
\thanks{}

\begin{abstract}
Let $A$ be an arbitrary $2\times 2$ matrix. In \cite{Cisneros:PhD,Cisneros:I2x2M} I gave a formula for the trace of the $k$-th symmetric power of $A$ in terms of the anti-diagonal entries of $A^{k+1}$ and $A$.
This was based on formulae that I found for the entries of the $k$-th power $A^k$ of the matrix $A$ in terms of its entries but I only sketched the idea of how I obtained such formulae.
In this article I give the full proof of those formulae by counting some walks of length $k$ over the complete digraph of order $2$. I compare them with formulae for $A^k$ given by Mc Laughlin in \cite{McLaughlin:CIDnP2x2M} and by Williams in \cite{Williams:nthP2x2M}. This leads to combinatorial identities, in particular expressions for Fibonacci and Lucas numbers.
\end{abstract}

\maketitle

%%%%%%%%%  SECTION

\section{Introduction}

Let $A=\smatr{a}{b}{c}{d}$ be a $2\times2$ matrix with complex entries and denote its $k$-th power by $A^{k}=\smatr{a_{k}}{b_{k}}{c_{k}}{d_{k}}$.
The matrix $A$ acts naturally by matrix multiplication on $\C^2$ giving a linear map $A\colon\C^2\to\C^2$.
The $k$-th symmetric power $S^k \C^2$ is isomorphic to the vector space of homogeneous polynomials of degree $k$ in two variables $x$ and $y$.
The matrix $A$ induces a linear map $S^kA\colon S^k \C^2\to S^k \C^2$ (see Section~\ref{sec:trace}).
In 1998 in my PhD thesis \cite[Proposition~4.22]{Cisneros:PhD} I proved as a collateral result that, if $b\neq0$ and $c\neq0$, the trace of the linear map $S^kA$ is given by (see Theorem~\ref{thm:fk.sk})
\begin{equation}\label{eq:TSk.bkb}
 \trace S^kA=b_{k+1}/b=c_{k+1}/c.
\end{equation}
This implies that the quotients $b_{k+1}/b=c_{k+1}/c$ are invariant under conjugation of $A$.
I proved this result giving explicit closed formulae for $\trace S^kA$ and the entries of $A^k$ in terms of the entries of $A$ (\cite[Lemma~4.20 \& Lemma~4.21]{Cisneros:PhD} see Lemma~\ref{lem:charma} and Theorem~\ref{lem:power}) and comparing them. After showing the result, several people gave me three more different proofs which I publish in 2005  in \cite{Cisneros:I2x2M}.
Being my original proof a pedestrian one compared with the other three more conceptual proofs, I gave the formulae for $A^k$ without proof.
Of course the veracity of such formulae can be proved by induction (this was done in the bachelor thesis \cite[Lema~4.2]{Diaz:IM2x2}), but this does not explain how the formulae were obtained.

Unaware of my article \cite{Cisneros:I2x2M}, Larcombe proved\footnote{Curiously he actually also gives four proofs. See also \cite{Larcombe-Fennessey:ADRITNP,Larcombe-Fennessey:SGTPMADRI}} in \cite[Theorem~1.1]{Larcombe:IGMADRIMP} that
%\begin{theorem}[{\cite[Theorem~1.1]{Larcombe:IGMADRIMP}}]
given a matrix $A=\smatr{a}{b}{c}{d}$ with $b\neq0$ and $c\neq0$ and denoting
$A^{k}=\smatr{a_{k}}{b_{k}}{c_{k}}{d_{k}}$ as before, one has
\begin{equation}\label{thm:larcombe.b/c}
b_k/c_k=b/c,\quad\text{for any $k\geq1$.}
\end{equation}
%\end{theorem}
However, $b_k/c_k=b/c$ is not invariant under conjugation of $A$.
This is an immediate corollary of \eqref{eq:TSk.bkb} which I had not noticed until Prof.~Larcombe pointed it out to me in a private communication in 2017. He also mentioned that \eqref{thm:larcombe.b/c} also follows from a formula for the $k$-th power of a $2\times2$ matrix given by  Mc Laughlin in \cite{McLaughlin:CIDnP2x2M}.
Prof.~Larcombe was also interested in my formulae for $A^k$, he and a former PhD student of his checked them using computer software for powers up to and beyond $500$
(see \cite[Appendix~B]{Larcombe-Fennessey:TRIP2M}), and encouraged me to publish them.

In the present article I finally give the full proof to obtain the formulae for $A^k$ in terms of the entries of $A$ given in \cite{Cisneros:PhD,Cisneros:I2x2M}.
This is done in Section~\ref{sec:An} where I also present the formulae for $A^k$ given by Mc Laughlin in \cite{McLaughlin:CIDnP2x2M} and by Williams in \cite{Williams:nthP2x2M}.
In Section~\ref{sec:trace} using \eqref{eq:TSk.bkb} we express the entries of $A^k$ in terms of the entries of $A$ and $\trace S^k$ and $\trace S^{k-1}$ (see Theorem~\ref{thm:Ak.trSk}). This gives the connection with Mc Laughlin's formulae. In Section~\ref{sec:identities} we obtain several combinatorial identities from the proof of Theorem~\ref{lem:power} and comparing the different formulae presented for $A^k$.
As an application we obtain (I hope) new formulae for Fibonacci and Lucas numbers together with some known ones.

\section{The $n$-th power of a $2\times2$ matrix}\label{sec:An}

In this section, given a $2\times2$ matrix $A$ we derive the explicit
formulae for the $n$-th power $A^n$ of $A$ in terms of its entries given in \cite{Cisneros:PhD,Cisneros:I2x2M}.
This is done using an idea of Felipe González of counting
walks of length $n$ over the complete digraph of order $2$ (Figure~\ref{fig}).
We also give the formulae for $A^k$ given by Mc Laughlin in \cite{McLaughlin:CIDnP2x2M} and by Williams in \cite{Williams:nthP2x2M} in order to compare them in the following sections.

\begin{theorem}\label{lem:power}
Consider the matrix $A=\smatr{a}{b}{c}{d}$. Denote by $a_{n}$,
$b_{n}$, $c_{n}$ and $d_{n}$ the corresponding entries of the matrix
$A^{n}$, i.e. $A^{n}=\smatr{a_{n}}{b_{n}}{c_{n}}{d_{n}}$. Then
\begin{align*}
a_{n}&=a^{n}+\sum_{s=1}^{\lfloor\frac{n}{2}\rfloor}\sum_{m=0}^{n-2s}\binom{n-s-m}{s}\binom{m+s-1}{m}a^{n-2s-m}b^{s}c^{s}d^{m}\\
b_{n}&=\sum_{s=0}^{\lfloor\frac{n-1}{2}\rfloor}\sum_{m=0}^{n-2s-1}\binom{n-s-m-1}{s}\binom{m+s}{m} a^{n-2s-m-1}b^{s+1}c^{s}d^{m}\\
c_{n}&=\sum_{s=0}^{\lfloor\frac{n-1}{2}\rfloor}\sum_{m=0}^{n-2s-1}\binom{n-s-m-1}{s}\binom{m+s}{m} a^{n-2s-m-1}b^{s}c^{s+1}d^{m}\\
d_{n}&=d^{n}+\sum_{s=1}^{\lfloor\frac{n}{2}\rfloor}\sum_{m=0}^{n-2s}\binom{n-s-m-1}{s-1}\binom{m+s}{m} a^{n-2s-m}b^{s}c^{s}d^{m}
\end{align*}
where $\lfloor x\rfloor$ denotes the floor function of $x$.
\end{theorem}

\begin{proof}
The entries $a_n$, $b_n$, $c_n$ and $d_n$ of $A^n$ are polynomials of
degree $n$ on $a$, $b$, $c$ and $d$. 
Since $A^{n}=A^{n-1}A=\smatr{a_{n-1}}{b_{n-1}}{c_{n-1}}{d_{n-1}}\smatr{a}{b}{c}{d}$
one gets the recursive equations
\begin{align}
a_{n}&=aa_{n-1}+cb_{n-1}\label{1}\\
b_{n}&=ba_{n-1}+db_{n-1}\label{2}\\
c_{n}&=ac_{n-1}+cd_{n-1}\label{3}\\
d_{n}&=bc_{n-1}+dd_{n-1}\label{4}.
\end{align}
Using these equations, one can find which kind of terms appear in the
entries of $A^{n}$.

\paragraph{\textbf{The entry $a_n$}} Given an arbitrary term in $a_n$, by
(\ref{1}) we have two possibilities, either it is of the form
$aa_{n-1}$ or it is of the form $cb_{n-1}$. On the other hand, given
an arbitrary term in $b_n$, by (\ref{2}) it is either of the form
$ba_{n-1}$ or $db_{n-1}$. We can represent the relation between
equations (\ref{1}) and (\ref{2}) with the digraph in Figure~\ref{fig},
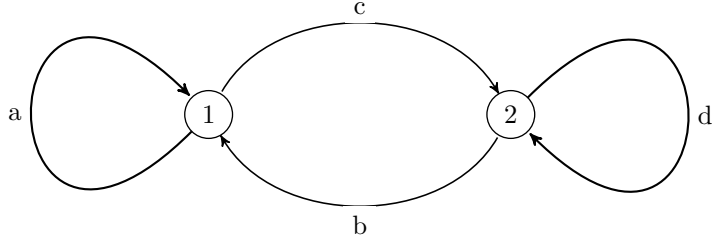
\begin{figure}[H]
\begin{center}
%\renewcommand*{\VertexLightFillColor}{black}
% \begin{tikzpicture}
% \SetGraphUnit{3}
% \Vertex[LabelOut,Lpos=90,L=$1$]{A}
% \EA[LabelOut,Lpos=90,L=$2$](A){B}
% \Edge[style={bend left,->},label=$c$](A)(B)
% \Edge[style={bend left,->},labelstyle={},label=$b$](B)(A)
% \Loop[dist=3cm,dir=WE,style={thick},label=$a$](A)
% \Loop[dist=3cm,dir=EA,style={thick},label=$d$](B)
% \end{tikzpicture}
\begin{tikzpicture}
\useasboundingbox (-2.5,-2) rectangle (7.5,2);
\grProb{1}{2}{$c$}{$b$}{$a$}{$d$}
\end{tikzpicture}
\end{center}
\caption{Complete digraph of order $2$.}\label{fig}
\end{figure}
\noindent where vertex 1 corresponds to equation (\ref{1}) (terms in $a_k$),
vertex 2 corresponds to equation (\ref{2}) (terms in $b_k$) and the
labels on the arcs indicate the factor we are factorising from the
term in $a_k$ (vertex 1) or $b_k$ (vertex 2) as we explain in the
following algorithm to find the form of an arbitrary term in $a_n$.
\begin{enumerate}\setcounter{enumi}{-1}
\item Take an arbitrary term $w$ in $a_n$. Go to \ref{s1} with $r=n$.
\item \label{s1}Term in $a_r$, (in vertex 1). We have two
possibilities:
\begin{enumerate}
\item Factorise an $a$ and get a term in $a_{r-1}$, (follow loop
$a$ at vertex 1). Go back to \ref{s1}.
\item Factorise a $c$ and get a term in $b_{r-1}$, (follow arc $c$
from vertex 1 to vertex 2). Go to \ref{s2}.
\end{enumerate}
\item \label{s2}Term in $b_r$, (in vertex 2). We have two
possibilities:
\begin{enumerate}
\item Factorise a $b$ and get a term in $a_{r-1}$, (follow arc $b$
from vertex 2 to vertex 1). Go to \ref{s1}.
\item Factorise a $d$ and get a term in $b_{r-1}$, (follow loop
$d$ at vertex 2). Go back to \ref{s2}.
\end{enumerate}
\end{enumerate}
Following the algorithm we shall reach a step when we get a term in
$a_1$ or $b_1$.
\begin{itemize}
\item If we get a term in $a_1$, we are at vertex 1 and since $a_1=a$,
factorising this last $a$ we finish again at vertex 1.
\item If we get a term in $b_1$, we are at vertex 2 and since $b_1=b$,
factorising this last $b$ we finish at vertex 1.
\end{itemize}
In both cases we finish at vertex 1. Since we started with an
arbitrary term $w$ in $a_n$ (vertex 1), the term $w$ corresponds to a
closed walk of length $n$ in the digraph. Hence we have the same
number of $b$'s and $c$'s and therefore the terms in $a_n$ have the
form
\begin{equation*}
a^rb^sc^sd^m,
\end{equation*}
with $r+2s+m=n$. The term $a^rb^sc^sd^m$ corresponds to closed walks
of length $n$ 
on the digraph starting and finishing at vertex 1 passing $r$ times
through the loop $a$, $s$ times through each of the arcs $b$ and $c$
and $m$ times through the loop $d$. Note that if $m\geq1$, then
$s\geq1$. Denote by $K_{(r,s,m)}$ the number of the different walks of
length $n$ with these characteristics, that is, which correspond to
the term $a^rb^sc^sd^m$. The number $K_{(r,s,m)}$ is the number of
times that such term appears in $a_n$, i.e. it is the coefficient of
$a^rb^sc^sd^m$ in $a_n$. Therefore, we have that $a_n$ has the form
\begin{equation*}
a_n=\sum_{(r,s,m)} K_{(r,s,m)}a^rb^sc^sd^m,
\end{equation*}
where $(r,s,m)$ runs over all the triples of non-negative integers
such that $r+2s+m=n$ and if $s=0$ then $m=0$.

Now we compute the value of $K_{(r,s,m)}$. Suppose that
$s\geq1$, because if $s=0$, then $m=0$ and we have the term
$a^n$ which appears just one time. Firstly, we put the $r$ $a$'s in a row:
\begin{equation*}
\underbrace{\quad a\quad a\quad a\quad\dots\quad a\quad a\quad}_{r} 
\end{equation*}
we can put the $s$ pairs of $b$'s and $c$'s in $r+1$ places. We can do
this in $\binom{r+s}{s}$ different ways. Now, the only place where we
can put the $d$'s is between a $b$ and a $c$, so we have $s$ places to
put $m$ $d$'s. This can be done in $\binom{m+s-1}{m}$ different
ways. Therefore we have that
$K_{(r,s,m)}=\binom{r+s}{s}\binom{m+s-1}{m}$. 
Finally express $r$ in terms of $s$ and $m$ as $r=n-2s-m$ and note
that $1\leq s \leq \lfloor\frac{n}{2}\rfloor$ and $0\leq m\leq n-2s$. Adding all
the different terms with their corresponding 
coefficients we obtain the formula for $a_n$.
\vspace{12pt}

\paragraph{\textbf{The entry $b_n$}} Given an arbitrary term in $b_n$, by (\ref{2}) it is either of the form $ba_{n-1}$ or $db_{n-1}$.
On the other hand, given an arbitrary term in $a_n$, by (\ref{1}) we have two possibilities, either it is of the form $aa_{n-1}$ or it is of the form $cb_{n-1}$.
Using the digraph in Figure~\ref{fig} we use the following algorithm.
\begin{enumerate}\setcounter{enumi}{-1}
\item Take an arbitrary term $w$ in $b_n$. Go to \ref{sb1} with $r=n$.
\item \label{sb1}Term in $b_r$, (in vertex 2). We have two
possibilities:
\begin{enumerate}
\item Factorise a $d$ and get a term in $b_{r-1}$, (follow loop
$d$ at vertex 2). Go back to \ref{sb1}.
\item Factorise a $b$ and get a term in $a_{r-1}$, (follow arc $b$
from vertex 2 to vertex 1). Go to \ref{sb2}.
\end{enumerate}
\item \label{sb2}Term in $a_r$, (in vertex 1). We have two
possibilities:
\begin{enumerate}
\item Factorise an $a$ and get a term in $a_{r-1}$, (follow loop
$a$ at vertex 1). Go back to \ref{sb2} .
\item Factorise a $c$ and get a term in $b_{r-1}$, (follow arc $c$
from vertex 1 to vertex 2). Go to \ref{sb1}.
\end{enumerate}
\end{enumerate}
Following the algorithm we shall reach a step when we get a term in
$a_1$ or $b_1$.
\begin{itemize}
\item If we get a term in $a_1$, we are at vertex 1 and since $a_1=a$,
factorising this last $a$ we finish again at vertex 1.
\item If we get a term in $b_1$, we are at vertex 2 and since $b_1=b$,
factorising this last $b$ we finish at vertex 1.
\end{itemize}
In both cases we finish at vertex 1. Since we started with an
arbitrary term $w$ in $b_n$ (vertex 2), the term $w$ corresponds to a
walk of length $n$ from vertex 2 to vertex 1 in the digraph. Hence we have one more $b$ than the number of $c$'s and therefore the terms in $b_n$ have the form
\begin{equation*}
a^rb^{s+1}c^sd^m,
\end{equation*}
with $r+2s+m+1=n$. The term $a^rb^{s+1}c^sd^m$ corresponds to walks
of length $n$ on the digraph starting at vertex 2 and finishing at vertex 1 passing $r$ times
through the loop $a$, $s+1$ times through the arc $b$, $s$ times through the arc $c$ and $m$ times through the loop $d$.
Denote by $K_{(r,s,m)}$ the number of the different walks of
length $n$ with these characteristics, that is, which correspond to
the term $a^rb^{s+1}c^sd^m$. The number $K_{(r,s,m)}$ is the number of
times that such term appears in $b_n$, i.e. it is the coefficient of
$a^rb^{s+1}c^sd^m$ in $b_n$. Therefore, we have that $b_n$ has the form
\begin{equation*}
b_n=\sum_{(r,s,m)} K_{(r,s,m)}a^rb^{s+1}c^sd^m,
\end{equation*}
where $(r,s,m)$ runs over all the triples of non-negative integers
such that $r+2s+m+1=n$.

Now we compute the value of $K_{(r,s,m)}$. Firstly, we put the $m$ $d$'s in a row:
\begin{equation*}
\underbrace{\quad d\quad d\quad d\quad\dots\quad d\quad d\quad}_{m}
\end{equation*}
we can put the $s$ pairs of $b$'s and $c$'s in $m+1$ places. We can do
this in $\binom{m+s}{s}$ different ways. Now, the only place where we
can put the $a$'s is after the $b$'s, so we have $s+1$ places to
put $r$ $a$'s. This can be done in $\binom{r+s}{s}$ different
ways. Therefore we have that
$K_{(r,s,m)}=\binom{r+s}{s}\binom{m+s}{m}$.
Finally express $r$ in terms of $s$ and $m$ as $r=n-2s-m-1$ and note
that if $s=0$ we still have one $b$ to go from vertex 2 to vertex 1, and since $2s=n-1-r-m$, the maximum value of $s$ is when $r=m=0$,
so $0\leq s\leq \lfloor\frac{n-1}{2}\rfloor$ and given $s$ the maximum number of $d$'s is $n-2s-1$, so $0\leq m\leq n-2s-1$. Adding all
the different terms with their corresponding coefficients we obtain the formula for $b_n$.

For the entries $c_n$ and $d_n$, in the digraph of Figure~\ref{fig} vertex 1 corresponds to equation \eqref{3} and vertex 2 to equation \eqref{4}.
\vspace{12pt}

\paragraph{\textbf{The entry $c_n$}} It is analogous to the entry $b_n$ but this time counting walks of length $n$ from vertex 1 to vertex 2 in the digraph.
\vspace{12pt}

\paragraph{\textbf{The entry $d_n$}} Given an arbitrary term in $d_n$, by
(\ref{1}) we have two possibilities, either it is of the form
$dd_{n-1}$ or it is of the form $bc_{n-1}$. On the other hand, given
an arbitrary term in $c_n$, by (\ref{2}) it is either of the form
$cd_{n-1}$ or $ac_{n-1}$.
Using the digraph in Figure~\ref{fig} we use the following algorithm.
\begin{enumerate}\setcounter{enumi}{-1}
\item Take an arbitrary term $w$ in $d_n$. Go to \ref{s1} with $r=n$.
\item \label{s1}Term in $d_r$, (in vertex 2). We have two
possibilities:
\begin{enumerate}
\item Factorise an $d$ and get a term in $d_{r-1}$, (follow loop
$d$ at vertex 2). Go back to \ref{s1}.
\item Factorise a $b$ and get a term in $c_{r-1}$, (follow arc $b$
from vertex 2 to vertex 1). Go to \ref{s2}.
\end{enumerate}
\item \label{s2}Term in $c_r$, (in vertex 2). We have two
possibilities:
\begin{enumerate}
\item Factorise a $c$ and get a term in $d_{r-1}$, (follow arc $c$
from vertex 1 to vertex 2). Go to \ref{s1}.
\item Factorise a $a$ and get a term in $c_{r-1}$, (follow loop
$a$ at vertex 1). Go back to \ref{s2}.
\end{enumerate}
\end{enumerate}
Following the algorithm we shall reach a step when we get a term in
$d_1$ or $c_1$.
\begin{itemize}
\item If we get a term in $d_1$, we are at vertex 2 and since $d_1=d$,
factorising this last $d$ we finish again at vertex 2.
\item If we get a term in $c_1$, we are at vertex 1 and since $c_1=c$,
factorising this last $c$ we finish at vertex 2.
\end{itemize}
In both cases we finish at vertex 2. Since we started with an
arbitrary term $w$ in $d_n$ (vertex 2), the term $w$ corresponds to a
closed walk of length $n$ in the digraph. Hence we have the same
number of $b$'s and $c$'s and therefore the terms in $d_n$ have the
form
\begin{equation*}
a^rb^sc^sd^m,
\end{equation*}
with $r+2s+m=n$. The term $a^rb^sc^sd^m$ corresponds to closed walks
of length $n$
on the digraph starting and finishing at vertex 2 passing $r$ times
through the loop $a$, $s$ times through each of the arcs $b$ and $c$
and $m$ times through the loop $d$. Note that if $r\geq1$, then
$s\geq1$. Denote by $K_{(r,s,m)}$ the number of the different walks of
length $n$ with these characteristics, that is, which correspond to
the term $a^rb^sc^sd^m$. The number $K_{(r,s,m)}$ is the number of
times that such term appears in $d_n$, i.e. it is the coefficient of
$a^rb^sc^sd^m$ in $d_n$. Therefore, we have that $d_n$ has the form
\begin{equation*}
d_n=\sum_{(r,s,m)} K_{(r,s,m)}a^rb^sc^sd^m,
\end{equation*}
where $(r,s,m)$ runs over all the triples of non-negative integers
such that $r+2s+m=n$ and if $s=0$ then $r=0$.

Now we compute the value of $K_{(r,s,m)}$. Suppose that
$s\geq1$, because if $s=0$, then $r=0$ and we have the term
$d^n$ which appears just one time. Firstly, we put the $m$ $d$'s in a row:
\begin{equation*}
\underbrace{\quad d\quad d\quad d\quad\dots\quad d\quad d\quad}_{m}
\end{equation*}
we can put the $s$ pairs of $b$'s and $c$'s in $m+1$ places. We can do
this in $\binom{m+s}{m}$ different ways. Now, the only place where we
can put the $a$'s is between a $b$ and a $c$, so we have $s$ places to
put $r$ $a$'s. This can be done in $\binom{r+s-1}{s-1}$ different
ways. Therefore we have that
$K_{(r,s,m)}=\binom{r+s-1}{s-1}\binom{m+s}{m}$.
Finally express $r$ in terms of $s$ and $m$ as $r=n-2s-m$ and note
that $1\leq s \leq \lfloor\frac{n}{2}\rfloor$ and $0\leq m\leq n-2s$. Adding all
the different terms with their corresponding
coefficients we obtain the formula for $d_n$.
%
% The formulae for $c_n$ and $d_n$ are analogous to the
% formulae for $b_n$ and $a_n$ respectively using equations (\ref{3})
% and (\ref{4}) and in the digraph of Figure~\ref{fig} vertex 1 corresponds to equation \eqref{3} and vertex 2 to equation \eqref{4}.
\end{proof}

\begin{remark}\label{rem:a.d}
By the symmetry of matrix multiplication one would expect the coefficients of $a_n$ and $d_n$ have the same form.
They look different because in both cases we are taking the sums with respect to $s$ and $m$,
but in $a_n$ the ``leading'' entry is $a$ and in $d_n$ the ``leading'' entry is $d$. So in order to be able to compare them, in $d_n$
we have to take the sums with respect to $s$ and the exponent of $a$. So using the identity $\binom{n}{k}=\binom{n}{n-k}$ in both binomial coefficients and setting $l=n-2s-m$, we have that $m=n-2s-l$ and making the substitutions, the coefficients of $d_n$ are (compare with \cite[Appendix~B]{Larcombe-Fennessey:TRIP2M})
\begin{equation*}
\tbinom{n-s-m-1}{s-1}\tbinom{m+s}{m}=\tbinom{n-2s-m+s-1}{n-2s-m}\tbinom{m+s}{s}=\tbinom{l+s-1}{l}\tbinom{n-s-l}{s}
\end{equation*}
so now they have the same form as the coefficients of $a_n$ and the new expression for $d_n$ is:
\begin{equation*}
d_{n}=d^{n}+\sum_{s=1}^{\lfloor\frac{n}{2}\rfloor}\sum_{l=0}^{n-2s}
\binom{n-s-l}{s}\binom{l+s-1}{l} a^{l}b^{s}c^{s}d^{n-2s-l}.
\end{equation*}
\end{remark}

\begin{remark}
In 2009 it was published in \cite[Theorem~3.1.2]{Brualdi-Cvetkovic:CAMTIA} the following theorem.
Let $A=(a_{ij})$ be an $m\times m$ matrix, and let $D(A)$ be the complete digraph of order $m$. We assign to the edge from vertex $i$ to vertex $j$ the weight $a_{ij}$.
The \textit{weight of a walk} in $D(A)$ is defined to be the product of the weights of all edges of the walk.
For each positive integer $n$, the entry $a_{i,j}^{(n)}$ of $A^n$ in the $i$-th row and $j$-th column equals the sum of the weights of all walks in $D(A)$ of length $n$ from vertex $i$ to vertex $j$.

Theorem~\ref{lem:power} gives explicit formulae for the sum of the weights of all walks in $D(A)$ of length $n$ from vertex $i$ to vertex $j$ for an arbitrary $2\times 2$ matrix $A$.
The proof of Theorem~\ref{lem:power} gives an explanation of \cite[Theorem~3.1.2]{Brualdi-Cvetkovic:CAMTIA}: each walk of length $n$ from vertex $i$ to vertex $j$ in $D(A)$ corresponds to a monomial in the entry $a_{i,j}^{(n)}$ of $A^n$, and each time we follow an edge in that walk, it means that we can factorize the weight of that edge in the corresponding monomial.

The proof of Theorem~\ref{lem:power} can be, in theory, generalised for $m\times m$ matrices with $m>2$. There are two main problems to get the corresponding formulae: to find the form of the possible monomials appearing in the entries of $A^k$, and to count how many times a particular monomial appears, that is, to compute its coefficient, which correspond to count how many walks have the same weight.
The combinatorial problem gets complicated even for $m=3$.
% \todo[inline]{Maybe say the difficulty to generalize for $m>2$, to find all the different types of monomials and to count how many times they appear, which is equivalente to count the different walks that have the same edges, each the same number of times, but in different order.}
\end{remark}

\subsection{Mc Laughlin's formula}
%In \cite{Larcombe:IGMADRIMP} Larcombe also mentions that
Mc Laughlin in \cite[Theorem~1]{McLaughlin:CIDnP2x2M} gave the following formula for the $n$-th power of a $2\times2$ matrix. %from which \eqref{thm:larcombe.b/c} follows.

%\begin{theorem}[{\cite[Theorem~1]{McLaughlin:CIDnP2x2M}}]
Given a matrix $A=\smatr{a}{b}{c}{d}$, let $T=a+d$ denote its trace and $D=ad-bc$ its determinant. Let
\begin{equation}\label{eq:yk}
 y_n=\sum_{i=0}^{\lfloor n/2\rfloor}\binom{n-i}{i}T^{n-2i}(-D)^i.
\end{equation}
Then, for $n\geq1$,
\begin{equation}\label{eq:An.Mc}
 A^n=\begin{pmatrix}
      y_n-dy_{n-1} & by_{n-1}\\
      cy_{n-1} & y_n-ay_{n-1}
     \end{pmatrix}.
\end{equation}
%\end{theorem}
The proof uses the fact that $y_n$ satisfies the recurrence (see proof of \cite[Theorem~1]{McLaughlin:CIDnP2x2M})
\begin{equation}\label{eq:yk.rec}
 y_{n+1}=Ty_n-Dy_{n-1},\quad\text{with $y_0=1$ and $y_1=T$.}
\end{equation}
In \cite{Larcombe:NFRMLAD2MP} Larcombe gives the following closed formula for $y_n$:
\begin{equation}\label{eq:yk.lar}
 y_n=\frac{1}{2}\biggl(\frac{T}{2}\biggr)^{n}\frac{\bigl(1+\sqrt{1-4D/T^2}\bigr)^{n+1}-\bigl(1-\sqrt{1-4D/T^2}\bigr)^{n+1}}{\sqrt{1-4D/T^2}}.
\end{equation}

\subsection{Williams' formula}
In \cite{Williams:nthP2x2M} Williams gave the following formula for the $n$-th power of a $2\times2$ matrix $A$ with eigenvalues $\alpha$ and $\beta$:
\begin{equation}\label{eq:Ak.will}
 A^n=\begin{cases}
      \alpha^n\bigl(\frac{A-\beta I}{\alpha-\beta}\bigr)+\beta^n\bigl(\frac{A-\alpha I}{\beta-\alpha}\bigr), & \text{if $\alpha\neq\beta$,}\\
      \alpha^n\bigl(nA-(n-1)\alpha I\bigr), & \text{if $\alpha=\beta$.}
     \end{cases}
\end{equation}
In \cite{McLaughlin:CIDnP2x2M} Mc Laughlin expressed \eqref{eq:Ak.will} in terms of the entries of $A$ as follows. Let $T=a+d$ be the trace of $A$ and $D=ad-bc$ its determinant, then
the eigenvalues are given by $\alpha=\bigl(T+\sqrt{T^2-4D}\bigr)/2$ and $\beta=\bigl(T-\sqrt{T^2-4D}\bigr)/2$. In the case $\alpha\neq\beta$ define
\begin{equation}\label{eq:zn}
 z_n=\frac{\alpha^n-\beta^n}{\alpha-\beta}=\frac{1}{2^{n-1}}\sum_{s=0}^{\lfloor\frac{n-1}{2}\rfloor}\binom{n}{2s+1}T^{n-2s-1}(T^2-4D)^s,
\end{equation}
then
\begin{equation}\label{eq:An,WM}
 A^n=z_n A-z_{n-1}DI.
\end{equation}
Equation \eqref{eq:An,WM} also holds in the case $\alpha=\beta$, that is, when $T^2-4D=0$, if $z_n$ is assumed to have the value on the right of \eqref{eq:zn}.

\section{Trace of symmetric powers of a $2\times2$ matrix}\label{sec:trace}

Let $W$ be the vector space of $2\times 2$ matrices over a field $K$. Let $A\in W$, the matrix $A$ acts naturally by matrix multiplication on $K^2$.
The $k$-th symmetric power $S^k K^2$ is isomorphic to the vector space of homogeneous polynomials of degree $k$ in two variables $x$ and $y$.
The monomials
\begin{equation}\label{eq:Sk.basis}
\Pj{j}(x,y)=\mon{k-j}{j},\quad 0\leq j\leq k,
\end{equation}
give a basis for the vector space $S^kK^2$.

% Let $A=\smatr{a}{b}{c}{d}$ as in the introduction. The matrix $A$ acts naturally by matrix multiplication on $\C^2$. The $k$-symmetric power $S^k\C^2$ is isomorphic to the space of homogeneous polynomials of degree $k$ in two variables $x$ and $y$.

The $k$-th symmetric power $S^kA$ of $A$ is the matrix of the linear action of $A$ on $S^kK^2$ given by
\begin{equation}\label{eq:su2.action}
(A\cdot P)(z)=P(zA)
\end{equation}
where
\begin{equation*}
P\in S^kK^2,\quad A=\matr{a}{b}{c}{d},\quad z=(x,y)\quad\text{and}\quad zA=(ax+cy,bx+dy),
\end{equation*}
with respect to the basis of $S^kK^2$ given in \eqref{eq:Sk.basis}.

\begin{remark}\label{rk:S0.S1}
Note that $S^0K^2\cong K$ since homogeneous polynomials of degree $0$ in two variables are the constant functions. In this case by \eqref{eq:su2.action} we have that $S^0A=Id_K$ and $\trace S^0A=1$.
We also have $S^1K^2\cong K^2$ and $S^1A=A$.
\end{remark}

Using the basis \eqref{eq:Sk.basis} of $S^kK^2$ to compute the matrix $S^kA$ we get the following formula for $\trace S^kA$.

\begin{proposition}[{\cite[Lemma~3.1]{Cisneros:I2x2M}}]\label{lem:charma}
Let $A=\smatr{a}{b}{c}{d}$. Then
\begin{equation*}
\trace S^kA=\sum_{j=0}^{k}\sum_{i=0}^{\min\{k-j,j\}}\binom{k-j}{i}\binom{j}{j-i} a^{k-j-i}b^{i}c^{i}d^{j-i}.
\end{equation*}
\end{proposition}

The main result of \cite{Cisneros:I2x2M} gives $\trace S^kA$ in terms of the anti-diagonal entries of $A$ and $A^{k+1}$.

\begin{theorem}[{\cite[Proposition~2.2]{Cisneros:I2x2M}}]\label{thm:fk.sk}
Let $A=\smatr{a}{b}{c}{d}$ with $b\neq0$ and $c\neq0$. Denote its $k+1$-st power by $A^{k+1}=\smatr{a_{k+1}}{b_{k+1}}{c_{k+1}}{d_{k+1}}$. Then
\begin{equation*}
\trace S^kA=b_{k+1}/b=c_{k+1}/c.
\end{equation*}
\end{theorem}

In \cite{Cisneros:I2x2M} four proofs of Theorem~\ref{thm:fk.sk} are given. The second one, which is the original proof,
compares the formula for $\trace S^kA$ given in Proposition~\ref{lem:charma} with the formulae for
$b_{k+1}/b$ and $c_{k+1}/c$ obtained from Theorem~\ref{lem:power}.

\begin{remark}\label{rem:inv.conj}
Since $\trace S^kA$ is invariant under conjugation of $A$, $b_{k+1}/b$ is also invariant under conjugation of $A$. In fact, if we know that $\trace S^kA=b_{k+1}/b$ we get that $\trace S^kA=b_{k+1}/b=c_{k+1}/c$ by conjugating $A$ by $\smatr{0}{1}{1}{0}$.
\end{remark}

Theorem~\ref{thm:fk.sk} expreses $\trace S^kA$ in terms of the entries of $A$ and $A^{k+1}$. On the other hand, the following theorem expresses the entries of $A^k$ in terms of the entries of $A$ and $\trace S^kA$ and $\trace S^{k-1}A$, and shows that the connection between Mc Laughlin's formula \eqref{eq:An.Mc} for $A^k$ and the formula given by Theorem~\ref{lem:power} is precisely Theorem~\ref{thm:fk.sk}.

\begin{theorem}\label{thm:Ak.trSk}
Let $A=\smatr{a}{b}{c}{d}$. Denote its $k$-th power by $A^{k}=\smatr{a_{k}}{b_{k}}{c_{k}}{d_{k}}$.
Then
\begin{align*}
 a_{k}&=\trace S^kA-d\trace S^{k-1}A,\\
 b_k&=b\trace S^{k-1}A,\\
 c_k&=c\trace S^{k-1}A,\\
 d_k&=\trace S^kA-a\trace S^{k-1}A.\\
\end{align*}
Hence
\begin{equation}\label{eq:Ak.TSk}
A^k=\begin{pmatrix}
     \trace S^kA-d\trace S^{k-1}A & b\trace S^{k-1}A\\
     c\trace S^{k-1}A &\trace S^kA-a\trace S^{k-1}A
    \end{pmatrix}.
\end{equation}
\end{theorem}

\begin{proof}
First suppose $b\neq0$ and $c\neq0$.
The second and third equalities follow from Theorem~\ref{thm:fk.sk}. For the first and fourth equalities, consider the equations \eqref{2} and \eqref{3} with $n=k+1$ and divide them, respectively, by $b$ and $c$ to get
\begin{equation*}
\frac{b_{k+1}}{b}=\frac{b}{b}a_{k}+d\frac{b_{k}}{b},\qquad \frac{c_{k+1}}{c}=a\frac{c_{k}}{c}+\frac{c}{c}d_{k},
\end{equation*}
and apply Theorem~\ref{thm:fk.sk}.

If $b=c=0$, $A$ is the diagonal matrix $A=\smatr{a}{0}{0}{d}$ and $A^k=\smatr{a^{k}}{0}{0}{d^{k}}$. By \eqref{eq:su2.action} we have that in this case the polynomials $P_j(x,y)$ are eigenvectors of $S^kA$
with eigenvalues $a^{k-j}d^j$ with $0\leq j\leq k$, hence $S^kA$ is a diagonal matrix and we have
\begin{equation}\label{eq:tr.diag}
 \trace S^kA=\sum_{j=0}^ka^{k-j}d^j.
\end{equation}
Then we have
\begin{align*}
\trace S^kA-d\trace S^{k-1}A&=\sum_{j=0}^ka^{k-j}d^j-d\sum_{j=0}^{k-1}a^{k-j-1}d^j\\
&=\sum_{j=0}^ka^{k-j}d^j-\sum_{j=0}^{k-1}a^{k-j-1}d^{j+1}\\
\intertext{separating in the first sum $j=0$ and then taking $j=m+1$}
&=a^k+\sum_{m=0}^{k-1}a^{k-m-1}d^{m+1}-\sum_{j=0}^{k-1}a^{k-j-1}d^{j+1}\\
&=a^k
\end{align*}
since the two sums cancel out. An analogous computation is for $d_k$.

If $b\neq0$ and $c=0$ by Theorem~\ref{thm:fk.sk} and Theorem~\ref{lem:power} $\trace S^kA$ does not involve $b$ (see the third identity in Corollary~\ref{cor:Sk.id} with $c=0$) and we get again \eqref{eq:tr.diag} and $a_k$ and $d_k$ are as in the previous diagonal case. The case $b=0$ and $c\neq0$ is analogous.
\end{proof}

The resemblance of \eqref{eq:An.Mc} and \eqref{eq:Ak.TSk} of course is not coincidence and we have:

\begin{theorem}\label{thm:yx=sk}
Let $A=\smatr{a}{b}{c}{d}$. Let $y_k$ given by \eqref{eq:yk}.
Then
\begin{equation*}
y_k=\trace S^kA.
\end{equation*}
\end{theorem}

\begin{proof}
First suppose $b\neq0$ and $c\neq0$. The proof follows from the second part of the fourth proof of Theorem~\ref{thm:fk.sk} in \cite{Cisneros:I2x2M} and \eqref{eq:yk.rec}.
Let $T=a+d$ be the trace of $A$ and $D=ad-bc$ its determinant. The characteristic polynomial of $A$ is given by $\lambda^2-T\lambda+D=0$. By Cayley-Hamilton Theorem and multiplying by $A^k$ we have
$A^{k+2}-T A^{k+1}+D A^k=0$. From the anti-diagonal entries we get
\begin{equation}\label{eq:bk.rec}
\begin{split}
\frac{b_{k+2}}{b}&=T\frac{b_{k+1}}{b}-D\frac{b_k}{b},\qquad\text{with $\frac{b_1}{b}=1$ and $\frac{b_2}{b}=T$.}\\
\frac{c_{k+2}}{c}&=T\frac{c_{k+1}}{c}-D\frac{c_k}{c},\qquad\text{with $\frac{c_1}{c}=1$ and $\frac{c_2}{c}=T$.}\\
\end{split}
\end{equation}
By Theorem~\ref{thm:fk.sk} $\trace S^kA=b_{k+1}/b=c_{k+1}/c$ thus from \eqref{eq:bk.rec} we get
\begin{equation}\label{eq:Sk.rec}
 \trace S^{k+1}A=T\trace S^kA-D\trace S^{k-1}A,
\end{equation}
with $\trace S^{0}A=1$ and $\trace S^1A=T$ (see Remark~\ref{rk:S0.S1}). From \eqref{eq:Sk.rec} and  \eqref{eq:yk.rec} we have that $\trace S^kA$ and $y_k$ satisfy the same recurrence, so they must be equal.

If $b=0$, $c=0$ or $b=c=0$ we have that the trace of $A$ is $T=a+d$ and its determinant $D=ad$. We saw in the proof of Theorem~\ref{thm:Ak.trSk} that in this case $\trace S^kA$ is given by \eqref{eq:tr.diag}. On the other hand, by \eqref{eq:yk} we have that
\begin{equation}\label{eq:yk.b0}
y_k=\sum_{i=0}^{\lfloor k/2\rfloor}\binom{k-i}{i}(a+d)^{k-2i}(-ad)^i,
\end{equation}
so we need to prove that they coincide.

Setting $u_k=\trace S^kA$, by the geometric series we have
\begin{equation}
u_k=\trace S^kA=\sum_{j=0}^ka^{k-j}d^j=\frac{a^{k+1}-d^{k+1}}{a-d}.
\end{equation}
Since $a$ and $d$ are the roots of the equation $z^2-(a+d)z+ad=0$ both individual sequences $a^k$ and $d^k$ satisfy the recurrence relation:
\begin{equation*}
a^k=(a+d)a^{k-1}-(ad)a^{k-2},\qquad d^k=(a+d)d^{k-1}-(ad)d^{k-2},
\end{equation*}
since this recurrence is linear $u_k$ also satisfies the recurrence relation
\begin{equation}\label{eq:rec.rel}
u_k=(a+d)u_{k-1}-(ad)u_{k-2},\quad\text{with $u_0=1$ and $u_1=a+d$.}
\end{equation}
By \eqref{eq:yk.rec} $y_k$ also satisfies the recurrence relation \eqref{eq:rec.rel}, so they must be equal.
An alternative proof of this case is to find a closed form for \eqref{eq:rec.rel} using the generating function $U(t) = \sum_{k=0}^{\infty} u_k t^k$. We have that
\begin{equation*}
U(t) - xtU(t) + yt^2U(t) = u_0 + (u_1 - xu_0)t + \sum_{k=2}^\infty (u_k - xu_{k-1} + yu_{k-2})t^k.
\end{equation*}
Setting $x=a+d, y=ad$, the summation term is 0. Thus:
\begin{equation}
U(t) = \frac{1 + (a+d - (a+d))t}{1 - (a+d)t + adt^2} =\frac{1}{1 - ((a+d)t - adt^2)}=\sum_{i=0}^\infty ((a+d)t - adt^2)^i.
\end{equation}
Using the Binomial Theorem we have
\begin{equation}
U(t) = \sum_{i=0}^{\infty} \sum_{j=0}^i \binom{i}{j} (a+d)^{i-j} (-ad)^j t^{i+j}.
\end{equation}
Setting $k = i+j$, the coefficient of $t^k$ is
\begin{equation}
u_k=\sum_{j=0}^{\lfloor k/2 \rfloor} \binom{k-j}{j} (a+d)^{k-2j} (-ad)^j
\end{equation}
which coincides with the value of $y_k$ given in \eqref{eq:yk.b0}.
\end{proof}

\begin{remark}
Since $\trace S^kA=y_k$ the quantity $y_k$ defined by Mc Laughlin is invariant under conjugation of $A$. This is not mentioned in the literature, but it is obvious from its definition in \eqref{eq:yk} since it is defined in terms of $T$ and $D$ which are quantities that are invariant under conjugation of $A$.
\end{remark}

\section{Combinatorial identities}\label{sec:identities}

In this section we obtain various combinatorial identities using the results of the two previous sections.

\subsection{Combinatorial identities from Theorem~\ref{lem:power}}

As a corollary of the proof of Theorem~\ref{lem:power} we have the following set of combinatorial identities.

\begin{proposition}
The following quantities are all equal:
\begin{align*}
2^{n-1}&=1+\sum_{s=1}^{[\frac{n}{2}]}\sum_{m=0}^{n-2s}
\binom{n-s-m}{s}\binom{m+s-1}{m}=1+\sum_{s=1}^{[\frac{n}{2}]}\binom{n}{2s}.\\
2^{n-1}&=\sum_{s=0}^{[\frac{n-1}{2}]}\sum_{m=0}^{n-2s-1}
\binom{n-s-m-1}{s}\binom{m+s}{m}=\sum_{s=0}^{[\frac{n-1}{2}]}\binom{n}{2s+1}.\\
2^{n-1}&=1+\sum_{s=1}^{[\frac{n}{2}]}\sum_{m=0}^{n-2s}
\binom{n-s-m-1}{s-1}\binom{m+s}{m}=1+\sum_{s=1}^{[\frac{n}{2}]}\binom{n}{2s}.
\end{align*}

\end{proposition}\label{prop:2n=binoms}
\begin{proof}
\textbf{First row of equalities.}
This is proved counting the total number of closed walks of
length $n$ starting and finishing at vertex 1.  
In the proof of Theorem~\ref{lem:power} we saw that the number of closed
walks of length $n$ on the digraph starting and finishing at vertex 1
passing $r$ times through the loop $a$, $s$ times through each of the
arcs $b$ and $c$ and $m$ times through the loop $d$ is
$K_{(r,s,m)}=\binom{r+s}{s}\binom{m+s-1}{m}$ when $s\geq1$. Hence, the total
number of closed 
walks of length $n$ starting and finishing at vertex 1 is given by the
expression on the middle, where the term $1$ corresponds to the
case $s=0$ which is the walk which 
passes $n$ times through the loop $a$. On the other hand, it is easy
to see that it is also $2^{n-1}$ since at each stage of the walk we
have two possibilities, either, take a loop and get back to the same
vertex, or go to the other vertex using an arc. We have only $n-1$
such choices, since the walk is closed, in the last one we have to go
to vertex 1. To see that this is also equal to the expression in the
right-hand-side note that the number of walks on the digraph of length
$n$ which pass $s$ times through each of the arcs $b$ and $c$ (not
caring about how many times they pass through the loops $a$ or $d$) is
$\binom{n}{2s}$ since we just need to choose from the $n$ places, $2s$
places ($s$ for $c$ and $s$ for $b$), the $b$'s and $c$'s alternate
starting with a $c$ and they automatically determine how many $a$'s and
$d$'s are: the spaces between a $c$ and a $b$ must be $d$'s and the
spaces between a $b$ and a $c$ must be $a$'s. Hence we have that
\begin{equation}\label{eq:ci.inproof}
\sum_{m=0}^{n-2s}\binom{n-s-m}{s}\binom{m+s-1}{m}=\binom{n}{2s}.
\end{equation}
The other two rows of equalities are proved in a similar way, but here we give alternative proofs.

\textbf{Second row of equalities.}
Consider the matrix $A=\smatr{1}{1}{1}{1}$, it has trace $T=2$ and determinant $D=0$.
Denote its $n$-th power by $A^n=A=\smatr{a_n}{b_n}{c_n}{d_n}$. We compare the expressions of $b_n$ (or $c_n$) obtained by the formulae of $A^n$ given in
Theorem~\ref{lem:power}, \eqref{eq:An.Mc} and \eqref{eq:An,WM}. Since $D=0$, the only term in \eqref{eq:yk} which is non-zero is when $i=0$, so $y_{n-1}=2^{n-1}$ and by \eqref{eq:An.Mc} we have that
$b_n=y_{n-1}=2^{n-1}$. On the other hand, by Theorem~\ref{lem:power}, $b_n$ equals the expression in the middle. Finally, since $D=0$, by \eqref{eq:An,WM} $b_n=z_n$ and substituting $T=2$ and $D=0$ in
\eqref{eq:zn} we obtain the third expression of the second row of equalities.

\textbf{Third row of equalities.}
One can also get the third row of equalities from the first one using Remark~\ref{rem:a.d}.
\end{proof}

\begin{remark}
Knowing from the middle row of Proposition~\ref{prop:2n=binoms} that $2^{n-1}=\sum_{s=0}^{[\frac{n-1}{2}]}\binom{n}{2s+1}$, the equality
\begin{equation*}
1+\sum_{s=1}^{[\frac{n}{2}]}\binom{n}{2s}=2^{n-1}
\end{equation*}
follows from the identity $\binom{n+1}{2s+1}=\binom{n}{2s+1}+\binom{n}{2s}$.
% \begin{align*}
% 2^{n}&=\sum_{s=0}^{[\frac{n}{2}]}\binom{n+1}{2s+1}=\sum_{s=0}^{[\frac{n}{2}]}\binom{n}{2s+1}+\sum_{s=0}^{[\frac{n}{2}]}\binom{n}{2s}\\
% 2^{n}&=2^{n-1}+\sum_{s=0}^{[\frac{n}{2}]}\binom{n}{2s}.
% \end{align*}
\end{remark}

\begin{corollary}\label{cor:2.2s}
\begin{align*}
\binom{n}{2s}&=\sum_{m=0}^{n-2s}\binom{n-s-m}{s}\binom{m+s-1}{m},\\
\binom{n}{2s+1}&=\sum_{m=0}^{n-2s-1}\binom{n-s-m-1}{s}\binom{m+s}{m}.\\
\end{align*}
\end{corollary}

\begin{proof}
The first equality was proved in Theorem~\ref{prop:2n=binoms} in \eqref{eq:ci.inproof}. The second one can be proved analogously but also in the following way.
Consider the matrix $A=\smatr{1}{b}{1}{1}$, it has trace $T=2$ and determinant $D=1-b$. Denote its $n$-th power by $A^n=\smatr{a_n}{b_n}{c_n}{d_n}$.
We compare the expressions of $b_n$ in Theorem~\ref{lem:power} and \eqref{eq:An,WM}.
By Theorem~\ref{lem:power} we have that
\begin{equation}\label{eq:Th.bn}
b_{n}=\sum_{s=0}^{\lfloor\frac{n-1}{2}\rfloor}\sum_{m=0}^{n-2s-1}\binom{n-s-m-1}{s}\binom{m+s}{m} b^{s+1}.
\end{equation}
On the other hand, by \eqref{eq:An,WM} we have that
\begin{align}
b_{n}&=bz_n=\frac{b}{2^{n-1}}\sum_{s=0}^{\lfloor\frac{n-1}{2}\rfloor}\binom{n}{2s+1}2^{n-2s-1}\bigl(4-4(1-b)\bigr)^s\notag\\
&=\sum_{s=0}^{\lfloor\frac{n-1}{2}\rfloor}\binom{n}{2s+1}b^{s+1}.\label{eq:W.bn}
\end{align}
Equating the coefficients of $b^{s+1}$ in \eqref{eq:Th.bn} and \eqref{eq:W.bn} we obtain the result.
\end{proof}

\begin{corollary}
\begin{multline*}
\sum_{m=0}^{n-2s}\binom{n-s-m}{s}\binom{m+s}{m}=\\
\sum_{m=0}^{n-2s-1}\Bigl[\binom{n-s-m-1}{s}\binom{m+s}{m}+\binom{n-s-m}{s}\binom{m+s-1}{m}\Bigr]+\binom{n-s-1}{n-2s}.
\end{multline*}
\end{corollary}

\begin{proof}
Follows from Corollary~\ref{cor:2.2s} and the identity $\binom{n+1}{2s+1}=\binom{n}{2s+1}+\binom{n}{2s}$.
\end{proof}

\subsection{Combinatorial identities from $\trace S^kA$}

By Theorem~\ref{thm:yx=sk} we have four ways to compute $\trace S^kA$.

\begin{corollary}\label{cor:Sk.id}
Given $A=\smatr{a}{b}{c}{d}$ we have that
\begin{align*}
\trace S^kA&=\sum_{i=0}^{\lfloor \frac{k}{2}\rfloor}(-1)^i\binom{k-i}{i}(a+d)^{k-2i}(ad-bc)^i\\
&=\sum_{j=0}^{k}\sum_{i=0}^{\min\{k-j,j\}}\binom{k-j}{i}\binom{j}{j-i} a^{k-j-i}b^{i}c^{i}d^{j-i}\\
&=\sum_{s=0}^{\lfloor\frac{k}{2}\rfloor}\sum_{m=0}^{k-2s}\binom{k-s-m}{s}\binom{m+s}{m} a^{k-2s-m}b^{s}c^{s}d^{m}\\
&=\tfrac{1}{2}\biggl(\frac{a+d}{2}\biggr)^{k}\tfrac{\bigl(1+\sqrt{1-4(ad-bc)/(a+d)^2}\bigr)^{k+1}-\bigl(1-\sqrt{1-4(ad-bc)/(a+d)^2}\bigr)^{k+1}}{\sqrt{1-4(ad-bc)/(a+d)^2}}.
\end{align*}
\end{corollary}

\begin{proof}
It follows by Theorem~\ref{thm:yx=sk}, Proposition~\ref{lem:charma}, \eqref{eq:yk} and \eqref{eq:yk.lar}.
\end{proof}

\subsection{Formulae for the Fibonacci numbers}

From Corollary~\ref{cor:Sk.id} we can get many combinatorial identities taking specific matrices $A$.
For instance, let $\{F_k\}_{k=0}^\infty$ denote the Fibonacci sequence, defined by $F_0=0$, $F_1=1$, $F_{k+1}=F_k+F_{k-1}$, for $k\geq1$. Let $A=\smatr{1}{1}{1}{0}$, then \cite[Theorem~32.1]{Koshy:FLNA}
\begin{equation}
 A^k=\begin{pmatrix}
      F_{k+1} & F_k\\
      F_k & F_{k-1}
     \end{pmatrix}
\end{equation}
and by  \eqref{eq:Ak.TSk} and substituting in the equalities of Corollary~\ref{cor:Sk.id} we recover the identities
\begin{equation}\label{eq:Fibonacci}
 F_{k}=\trace S^{k-1}A=\sum_{s=0}^{\lfloor\frac{k-1}{2}\rfloor}\binom{k-s-1}{s}=\frac{1}{\sqrt{5}}\Biggl[\biggl(\frac{1+\sqrt{5}}{2}\biggr)^{k}-\biggl(\frac{1-\sqrt{5}}{2}\biggr)^{k}\Biggr]
\end{equation}
given in \cite[Chapter~1 \S15]{Vorobev:FibonacciNumbers} and Binet's formula \cite[(1.20)]{Vorobev:FibonacciNumbers}. Also the entry $d_k$ in \eqref{eq:Ak.TSk} gives back the recursive relation $F_{k-1}=F_{k+1}-F_k$.

% This is just to record the argument
% For the second equality, since $d=0$, the non-zero terms are when $i=j$. So we have
% \begin{equation*}
% 0\leq i\leq \min\{k-j-1,j\}\Rightarrow 0\leq i\leq k-j-1\Rightarrow 0\leq k-j-i-1=k-2j-1,
% \end{equation*}
% thus we have $j\leq \frac{k-1}{2}$. Therefore we have
% \begin{equation*}
% \sum_{s=0}^{\lfloor\frac{k-1}{2}\rfloor}\binom{k-s-1}{s}.
% \end{equation*}

The important new observation is the first equality in \eqref{eq:Fibonacci}: \textbf{the $k$-th Fibonacci number is the trace of the matrix $S^{k-1}A$ and therefore it is invariant under conjugation of $A$} (see Remark~\ref{rem:inv.conj}). Let $B=\smatr{a}{b}{c}{d}$ a non-singular matrix, that is $ad-bc\neq0$. Conjugate $A=\smatr{1}{1}{1}{0}$ by $B$ obtaining
\begin{equation}
 C=BAB^{-1}=\begin{pmatrix}
             \frac{ad+bd-ac}{ad-bc} & \frac{a^2-b^2-ab}{ad-bc}\\
             \frac{d^2-c^2+cd}{ad-bc} & \frac{ac-bc-bd}{ad-bc}
            \end{pmatrix}.
\end{equation}
Then, by Corollary~\ref{cor:Sk.id} and \eqref{eq:Fibonacci} we get

\begin{corollary}
Let $a,b,c,d\in\mathbb{C}$ such that $ad-bc\neq0$. Then
\begin{multline}\label{eq:fibi.infty}
F_{k}=\frac{1}{(ad-bc)^{k-1}}\sum_{s=0}^{\lfloor\frac{k-1}{2}\rfloor}\sum_{m=0}^{k-2s-1}\binom{k-s-m-1}{s}\binom{m+s}{m}\cdot\\
(ad+bd-ac)^{k-2s-m-1}(a^2-b^2-ab)^s(d^2-c^2+cd)^s(ac-bc-bd)^m.
\end{multline}
\end{corollary}

In this way we get identities for the Fibonacci numbers, one for each invertible matrix $B=\smatr{a}{b}{c}{d}$.
% For instance, taking $B=\smatr{1}{1}{1}{2}$ we get $C=BAB^{-1}=\smatr{3}{-1}{5}{-2}$ and substituting in \eqref{eq:fibi.infty}
% \begin{equation*}
% F_{k}=\sum_{s=0}^{\lfloor\frac{k-1}{2}\rfloor}\sum_{m=0}^{k-2s-1}\binom{k-s-m-1}{s}\binom{m+s}{m}(-1)^{m+s}\cdot3^{k-2s-m-1}\cdot5^s\cdot2^m.
% \end{equation*}

You can obtain a formula in terms of your favorite non-zero complex number $\eta\neq0$ using for instance $B=\smatr{\eta}{1}{0}{1}$,
we get $C=BAB^{-1}=\frac{1}{\eta}\smatr{\eta+1}{\eta^2-\eta-1}{1}{-1}$ and by \eqref{eq:fibi.infty} we obtain the following corollary (compare with \cite[Corollary~2]{McLaughlin:CIDnP2x2M}).

\begin{corollary}
Let $\eta$ be any non-zero complex number. Then
\begin{equation}\label{eq:Fn.eta}
F_{k}=\frac{1}{\eta^{k-1}}\sum_{s=0}^{\lfloor\frac{k-1}{2}\rfloor}\sum_{m=0}^{k-2s-1}\tbinom{k-s-m-1}{s}\tbinom{m+s}{m}(-1)^{m}(\eta+1)^{k-2s-m-1}(\eta^2-\eta-1)^s.
\end{equation}
\end{corollary}

For example, if $\eta=\pi$ we get
\begin{equation*}
F_{k}=\frac{1}{\pi^{k-1}}\sum_{s=0}^{\lfloor\frac{k-1}{2}\rfloor}\sum_{m=0}^{k-2s-1}\tbinom{k-s-m-1}{s}\tbinom{m+s}{m}(-1)^{m}(\pi+1)^{k-2s-m-1}(\pi^2-\pi-1)^s.
\end{equation*}

One special case is of course when $\eta=\phi=\frac{1+\sqrt{5}}{2}\approx1.618033$, the golden ratio.

\begin{corollary}
\begin{equation*}
F_{k}=\sum_{m=0}^{k-1}(-1)^{m}\phi^{k-2m-1}.
\end{equation*}
\end{corollary}

\begin{proof}
The number $\phi$ satisfies the equation
\begin{equation}\label{eq:phi.eq}
 \phi^2-\phi-1=0,
\end{equation}
so the non-zero terms in
\eqref{eq:Fn.eta} are when $s=0$ obtaining
\begin{align*}
F_{k}&=\frac{1}{\phi^{k-1}}\sum_{m=0}^{k-1}(-1)^{m}(\phi+1)^{k-m-1}\\
&=\frac{1}{\phi^{k-1}}\sum_{m=0}^{k-1}(-1)^{m}\phi^{2(k-m-1)}\qquad\text{using \eqref{eq:phi.eq}}\\
&=\sum_{m=0}^{k-1}(-1)^{m}\phi^{k-2m-1}.\qedhere
\end{align*}
\end{proof}

Let $A=\smatr{1}{1}{1}{2}$, then \cite[Exercise~32.19]{Koshy:FLNA}
\begin{equation}
 A^k=\begin{pmatrix}
      F_{2k-1} & F_{2k}\\
      F_{2k} & F_{2k+1}
     \end{pmatrix}
\end{equation}
Then by Theorem~\ref{lem:power} we have
\begin{corollary}
\begin{equation*}
 F_{2k}=\sum_{s=0}^{\lfloor\frac{k-1}{2}\rfloor}\sum_{m=0}^{k-2s-1}\binom{k-s-m-1}{s}\binom{m+s}{m}2^m.
\end{equation*}
\end{corollary}

\subsection{Formulae for Fibonacci and Lucas numbers}

Let $\{L_k\}_{k=0}^\infty$ denote the Lucas sequence, defined by $L_0=2$, $L_1=1$, $L_{k+1}=L_k+L_{k-1}$, for $k\geq1$. Let $A=\smatr{\frac{1}{2}}{\frac{5}{2}}{\frac{1}{2}}{\frac{1}{2}}$, then \cite[Corollary~2]{Demirturk:FLSMM}
\begin{equation*}
A^k=\begin{pmatrix}
      \frac{1}{2}L_{k} & \frac{5}{2}F_{k}\\[3pt]
      \frac{1}{2}F_{k} & \frac{1}{2}L_{k}
     \end{pmatrix}
\end{equation*}
We get the following expression for $F_k$.
\begin{corollary}\label{cor:YAFk}
\begin{equation*}
F_{k}=\frac{1}{2^{k-1}}\sum_{s=0}^{\lfloor\frac{k-1}{2}\rfloor}\sum_{m=0}^{k-2s-1}\binom{k-s-m-1}{s}\binom{m+s}{m}5^s
\end{equation*}
\end{corollary}

\begin{proof}
Follows from the third equality in Corollary~\ref{cor:Sk.id} or the second equality in Theorem~\ref{lem:power}.
\end{proof}

We also obtain the following expressions for $L_k$, the second one is given in equation (92) in \cite[p.~69]{Vajda:FLNGS}.
\begin{corollary}\label{cor:Lk}
\begin{align*}
L_k&=\frac{1}{2^{k-1}}+\frac{1}{2^{k-1}}\sum_{s=1}^{\lfloor\frac{k}{2}\rfloor}\sum_{m=0}^{k-2s}\binom{k-s-m}{s}\binom{m+s-1}{m}5^s.\\
&=\frac{1}{2^{k-1}}\sum_{s=0}^{\lfloor\frac{k}{2}\rfloor}\binom{k}{2s}5^s.
\end{align*}
\end{corollary}

\begin{proof}
The first equality follows from Theorem~\ref{lem:power}, and the second one using the first equality in Corollary~\ref{cor:2.2s}.
\end{proof}

\begin{remark}
Note that by Theorem~\ref{thm:fk.sk} the expression for $F_k$ in Corollary~\ref{cor:YAFk} is invariant under conjugation, and we could also obtain it from \eqref{eq:fibi.infty} taking $a=1$, $b=2$, $c=1$ and $d=0$. On the other hand, the expressions for $L_k$ in Corollary~\ref{cor:Lk} \textbf{are not} invariant under conjugation.
\end{remark}

\noindent\textbf{Acknowledgements} 
I would like to thank Prof.~Peter J. Larcombe for the enlightening private communications, for sharing his results with me, for making me aware of Mc Laughlin's article and his interest in my work.

% \bibliography{matart,mypapers,matbook,thesis}
% \bibliographystyle{plain}
\end{document}